\documentclass[a4paper, draft, 12pt]{amsproc}
\usepackage{amssymb}
\usepackage{amsmath,amssymb,ulem,color}

\usepackage[hyphens]{url} 
\usepackage[dvips]{graphicx} 
\usepackage{cmdtrack} 
\usepackage{mathrsfs}                       

\theoremstyle{plain}
 \newtheorem{theorem}{Theorem}[section]
 
 \newtheorem{lem}{Lemma}[section]
 
\theoremstyle{Definition}
 \newtheorem{exm}{Example}[section]
 
 \newtheorem{dfn}{Definition}[section]
\theoremstyle{remark}

 \numberwithin{equation}{section}

\renewcommand{\setminus}{\smallsetminus}
\title[A Study of Fatou Set, Julia set and Escaping Set in Conjugate...]{A Study of Fatou Set, Julia set and Escaping Set in Conjugate Transcendental Semigroup}



\author[B. H. Subedi]{\bfseries  Bishnu Hari Subedi}

\address{ 
Central Department of Mathematics \\ 
Institute of Science and Technology   \\ 
Tribhuvan University   \\ 
Kirtipur, Kathmandu\\
Nepal}
\email{subedi.abs@gmail.com / subedi\_bh@cdmathtu.edu.np }

\author[A. Singh]{Ajaya Singh}
\address{Central Department of Mathematics, Institute of Science and Technology, Tribhuvan University, Kirtipur, Kathmandu, Nepal }
\email{singh.ajaya1@gmail.com / singh\_a@cdmathtu.edu.np} 
\thanks{This research work of first author is supported by the PhD faculty fellowship of University Grants Commission, Nepal} 

\begin{document}

{\begin{flushleft}\baselineskip9pt\scriptsize
\end{flushleft}}
\vspace{18mm} \setcounter{page}{1} \thispagestyle{empty}

\begin{abstract}
We define commutator of a transcendental semigroup, and on the basis of this concept, we define conjugate semigroup. We prove that the conjugate semigroup is nearly abelian if and only if  the given semigroup is nearly abelian. We also prove that image of each of escaping set, Julia set and Fatou set  under commutator (affine complex conjugating maps) is equal to respectively escaping set, Julia set and Fatou set of conjugate semigroup. Finally, we prove that every element of the nearly abelian semigroup $ S $ can be written as the composition of an element  from the set generated by the  set of commutators $ \Phi(S) $ and the composition of the certain powers of its generators.
\end{abstract}

\maketitle
\section{Introduction}
Throughout this paper, we denote the \textit{complex plane} by $\mathbb{C}$ and set of integers greater than zero by $\mathbb{N}$. 
We assume the function $f:\mathbb{C}\rightarrow\mathbb{C}$ is \textit{transcendental entire function} (TEF) unless otherwise stated. 
For any $n\in\mathbb{N}, \;\; f^{n}$ always denotes the nth \textit{iterates} of $f$. Let $ f $ be a TEF. The set of the form
$$
I(f) = \{z\in \mathbb{C}:f^n(z)\rightarrow \infty \textrm{ as } n\rightarrow \infty \}
$$
is called an \textit{escaping set} and any point $ z \in I(S) $ is called \textit{escaping point}. For TEF $f$, the escaping set $I(f)$ was first studied by A. Eremenko \cite{ere}. He showed that 
 $I(f)\not= \emptyset$; the boundary of this set is a Julia set $ J(f) $ (that is, $ J(f) =\partial I(f) $);
 $I(f)\cap J(f)\not = \emptyset$; and 
 $\overline{I(f)}$ has no bounded component.  Note that the complement of Julia set $ J(f) $ in complex plane $ \mathbb{C} $ is a \textit{Fatou set} $F(f)$.
 
We confine our study on Fatou set, Julia set and escaping set of transcendental semigroup and its conjugate semigroup.  It is very obvious fact that a set of transcendental entire maps on $ \mathbb{C} $  naturally forms a semigroup. Here, we take a set $ A $ of transcendental entire maps and construct a semigroup $ S $ consists of all elements that can be expressed as a finite composition of elements in $ A $. We call such a semigroup $ S $ by \textit{transcendental semigroup} generated by set $ A $. 
Our particular interest is to study of the dynamics of the families of transcendental entire maps.  For a collection $\mathscr{F} = \{f_{\alpha}\}_{\alpha \in \Delta} $ of such maps, let 
$$
S =\langle f_{\alpha} \rangle
$$ 
be a \textit{transcendental semigroup} generated by them. The index set $ \Delta $ to which $ \alpha $  belongs is allowed to be infinite in general unless otherwise stated. 
Here, each $f \in S$ is a transcendental entire function and $S$ is closed under functional composition. Thus, $f \in S$ is constructed through the composition of finite number of functions $f_{\alpha_k},\;  (k=1, 2, 3,\ldots, m) $. That is, $f =f_{\alpha_1}\circ f_{\alpha_2}\circ f_{\alpha_3}\circ \cdots\circ f_{\alpha_m}$. 

A semigroup generated by finitely many transcendental functions $f_{i}, (i = 1, 2, \\  \ldots, n) $  is called \textit{finitely generated transcendental semigroup}. We write $S= \langle f_{1}, f_{2},\\  \ldots,f_{n} \rangle$. The transcendental semigroup $S$ is \textit{abelian} if  $f_i\circ f_j =f_j\circ f_i$  for all generators $f_{i}$ and $f_{j}$  of $ S $. The semigroup $ S $ is right cancellative if $ f \circ g = h \circ g \Longrightarrow f = h $, left cancellative if $ f \circ g = h \circ h \Longrightarrow g = f$ for all $ f, g, h \in S $ and cancellative if it is both right and left cancellative. 

The family $\mathscr{F}$  of complex analytic maps forms a \textit{normal family} in a domain $ D $ if given any composition sequence $ (f_{\alpha}) $ generated by the member of  $ \mathscr{F} $,  there is a subsequence $( f_{\alpha_{k}}) $ which is uniformly convergent or divergent on all compact subsets of $D$. If there is a neighborhood $ U $ of the point $ z\in\mathbb{C} $ such that $\mathscr{F} $ is normal family in $U$, then we say $ \mathscr{F} $ is normal at $ z $. If  $\mathscr{F}$ is a family of members from the transcendental semigroup $ S $, then we simply say that $ S $ is normal in the neighborhood of $ z $ or $ S $ is normal at $ z $. 
Semigroup $ S $ is \textit{iteratively divergent} at $ z $ if $f^n(z)\rightarrow \infty \; \textrm{as} \; n \rightarrow \infty$ for all $ f \in S $. 

Based  on the Fatou-Julia-Eremenko theory of a complex analytic function, the Fatou set, Julia set and escaping set in the settings of transcendental  semigroup are defined as follows.
\begin{dfn}[\textbf{Fatou set, Julia set and escaping set}]\label{2ab} 
\textit{Fatou set} of the transcendental semigroup $S$ is defined by
  $$
  F (S) = \{z \in \mathbb{C}: S\;\ \textrm{is normal in a neighborhood of}\;\ z\}
  $$
and the \textit{Julia set} $J(S) $ of $S$ is the compliment of $ F(S) $ where as the escaping set of $S$ is defined by 
$$
I(S)  = \{z \in \mathbb{C}: S \;  \text{is iteratively divergent at} \;z \}.
$$
We call each point of the set $  I(S) $ by \textit{escaping point}.        
\end{dfn} 

There is a slightly larger family of transcendental  semigroups that can  fulfill most of the results of abelian transcendental semigroup. We call these semigroups nearly abelian and it is considered the more general form than that of abelian semigroups.
\begin{dfn}[\textbf{Nearly abelian semigroup}] \label{1p}
We say that a transcendental semigroup $ S $ is \textit{nearly abelian} if there is a  family  $ \Phi = \{\phi_{i} \} $ of  conformal maps  such that
\begin{enumerate}
\item $ \phi_{i}(F(S)) = F(S) $ for all $ \phi_{i}\in \Phi $ and
\item for all $ f, g \in S $, there is a $ \phi \in \Phi $ such that $ f \circ g = \phi \circ g\circ f  $.
\end{enumerate}
\end{dfn}
\begin{dfn}[\textbf{Commutator}]\label{com}
Let $ S $ be a transcendental semigroup. The set of the form
$$\Phi(S) = \{\phi:\; \text{there are}\;\ f,  g\in S\;\ \text{such that}\;\ f\circ g = \phi \circ g\circ f\}  $$ 
is called the set of \textit{commutators} of $ S $. We write $ \phi = [f, g] $ if $ f \circ g = \phi \circ g\circ f  $.  
\end{dfn}
The notion of commutator is very useful to obtain conjugate maps of each generator $ f_{i} $ of the semigroup $ S $ and conjugate semigroup of the semigroup $ S $.
\begin{dfn}[\textbf{conjugate semigroup}]\label{csg}
Let $S =\langle f_{1}, f_{2}, f_{3}, \ldots, f_{n} \rangle$ be a finitely generated transcendental semigroup and $ \Phi(S) $ be a set of its commutators. 
Let us define a set 
\begin{equation}\label{1t}
S^{'} = \langle \phi \circ  f_{1} \circ \phi^{-1}, \; \phi \circ f_{2} \circ \phi^{-1}, \ldots,  \; \phi \circ f_{n} \circ \phi^{-1} \rangle 
\end{equation} 
where $ \phi \in \Phi(S) $ such that $ \phi = [f_{i}, f_{j}]$ and $ \phi^{-1}  = [f_{j}, f_{i}]$ as we defined before.  If we let
$ g_{i} = \phi \circ  f_{i} \circ \phi^{-1} $, then we say function $ f_{i}$ is conjugate to  $g_{i}$ by a map $ \phi : \mathbb{C}\rightarrow \mathbb{C}$. The semigroup $ S^{'} $ is then called a \textit{conjugate semigroup} of semigroup $ S $. 
\end{dfn}
The image of the Fatou set, Julia set and escaping set of nearly abelian semigroup  under commutator $ \phi \in \Phi(S) $ is respectively Fatou set, Julia set and  escaping set of its conjugate.
\begin{theorem}\label{cgs2}
Let $S =\langle f_{1}, f_{2}, f_{3}, \ldots, f_{n} \rangle$ be a nearly abelian transcendental semigroup with commutators the maps of the form $\phi(z) = az + b $ for some non-zero $ a $ and $S^{'} = \langle \phi \circ  f_{1} \circ \phi^{-1},\;  \phi \circ f_{2} \circ \phi^{-1}, \ldots, \phi \circ f_{n} \circ \phi^{-1} \rangle$ be a conjugate semigroup of $ S $. Then $ \phi (I(S)) = I(S^{'}), \; \phi (J(S)) =  J(S^{'}) $ and $\phi (F(S)) = F(S^{'})$.
\end{theorem}

Analogous to {\cite [Theorem 4.3]{hin}}, every function of the nearly abelian  transcendental semigroup $ S $ can be written as the composition of either an element  of commutator $ \Phi(S) $ or an element from the set generated by $ \Phi(S) $ and the composition of the certain powers of its generators.

\begin{theorem}\label{cgs4}
Let  $S =\langle f_{1}, f_{2}, f_{3}, \ldots, f_{n},\ldots \rangle$ be a nearly abelian cancellative transcendental semigroup. Then every element $ f \in S $ can be written in the form 
$$
f = \phi \circ f^{t_{1}}_{1}\circ f^{t_{2}}_{2}\circ f^{t_{3}}_{3}\circ\cdots f^{t_{m}}_{m} 
$$
where $ \phi \in \Phi(S) $ if $  \Phi(S)  $ is a group or semigroup otherwise $  \phi \in  G$, where $ G =\langle\Phi(S) \rangle $  is a group generated by  $ \Phi(S) $ and $ t_{i} $ are non-negative integers.
\end{theorem}

\section{The Notion Conjugate Semigroup and  Proof of the Theorem \ref{cgs2}}

Let $ S $ be a transcendental semigroup.  If there is a holomorphic function $ \phi $ such that $f\circ g = \phi \circ g\circ f  $ for every pair of functions  $ f, g \in S $, then $ \phi $ is called commutator of $ f $ and $ g $. Note that such a commutator is unique for every pair of transcendental entire functions. 
Recall that 
$$\Phi(S) = \{\phi:  f\circ g = \phi \circ g\circ f\; \text{for every pair of functions}\; f,  g\in S \}$$
is a set of commutators of transcendental semigroup $ S $. If $ S $ is abelian, then commutator $ \phi $ is an identity function. 
As in definition \ref{com}, we write $ \phi =[f, g] $ if $ f\circ g = \phi \circ g\circ f$. Note that  $ [f, g]^{-1} = [g, f] $ and for any $ f\in S, \; [f, f]$ =  identity. So,  in  $\Phi(S)$, there is an identity element and inverse of each $ \phi\in\Phi(S) $.  It is not clear in general whether $\Phi(S)$ has group structure or semigroup structure but we can make a group or semigroup $ G = \langle \Phi(S) \rangle $ generated by the elements of $ \Phi(S) $ whenever it is necessary. Note that there are some commutator identities of groups which can be verified in $ \Phi(S) $ if $ S $ is a cancellative transcendental semigroup. For example:
\begin{enumerate}
\item $ [f,\;  g \circ f^{n}] = [f,\; g] $.
\item $[f,\; f^{n} \circ g] \circ f^{n} = f^{n} \circ [f,\; g]$.
\item $[f\circ g, \; g \circ f] \circ g \circ f = f \circ g \circ [g,\; f] $ 
\end{enumerate}
In practice, there is a commutator for given pair of transcendental entire functions. For example: 
\begin{exm}
Let $ f(z) = e^{z^{2}} + \lambda $ and $ g(z) = - f(z)$, where $ \lambda \in \mathbb{C} $. It is easy to see that  $ (f \circ g)(z) = e^{(z^{2} + \lambda )^{2}}+ \lambda = \phi(-e^{(z^{2} + \lambda)^{2}}- \lambda) = (\phi \circ g \circ f)(z)$, where $ \phi(z) = -z $.  Likewise,   if $f(z) = \lambda \cos z,  \; (\lambda \in \mathbb{C})  $ and $ g(z) = -f(z) $, then $ (f \circ g)(z) = \lambda \cos (\lambda \cos z) = \phi (-\lambda \cos (\lambda \cos z)) =(\phi \circ g \circ f)(z) $, where $ \phi(z) = -z $. 
\end{exm}

The family  $\Phi(S)$ of  holomorphic functions $ \phi $ for which $f \circ g = \phi \circ g\circ f$ for some  $ f, g \in S $ is assumed to be pre-compact. This means that any sequence $ (\phi_{i}) $ of elements of $ \Phi(S) $ contains a subsequence $ (\phi_{i_{k}}) $ that converges to holomorphic function (but not to a constant function) uniformly on $ \mathbb{C} $. 

On the basis of the notion of commutator, we can get  conjugate map of each generator $ f_{i} $ of the semigroup $ S $ and conjugate semigroup of the semigroup $ S $ as defined in above definition \ref{csg}.  If our semigroup is nearly abelian, our most of the tasks on commutators  will be more easy so that dynamics on such semigroups can be handled in a right track. We prove the following result that shows conjugate semigroup is nearly abelian if and only if semigroup itself is nearly abelian.

\begin{theorem}\label{csg1}
The conjugate  semigroup $S^{'} = \langle \phi \circ f_{1} \circ \phi^{-1},\;  \phi \circ f_{2} \circ \phi^{-1},\;  \ldots, \phi \circ f_{n} \circ \phi^{-1} \rangle$ of a transcendental semigroup $S =\langle f_{1}, f_{2}, f_{3}, \ldots, f_{n} \rangle$ is nearly abelian if and only if $ S $ is nearly abelian.
\end{theorem}
\begin{proof}
Let $ S $ is a nearly abelian transcendental semigroup. Then $f_{i}\circ f_{j} = \phi \circ f_{j}\circ f_{i}$ for all generators $ f_{i},\; f_{j}\in S $ and $ \phi \in \Phi (S) $. Now for any $  \phi \circ  f_{i} \circ \phi^{-1},\;   \phi \circ  f_{j} \circ \phi^{-1} \in S^{'}$, we have 
\begin{align*}
(\phi \circ  f_{i} \circ \phi^{-1}) \circ  (\phi \circ  f_{j} \circ \phi^{-1}) =  & \phi \circ  f_{i} \circ  f_{j} \circ \phi^{-1} \\
               =  &   \phi \circ \xi \circ f_{j} \circ    f_{i} \circ \phi^{-1}\;\; \text{for some}\;\; \xi \in \Phi(S) \\
               =  & \xi \circ \phi \circ f_{j} \circ    f_{i} \circ \phi^{-1}\\
                =  & \xi \circ (\phi \circ f_{j} \circ \phi^{-1}) \circ (\phi \circ f_{i} \circ \phi^{-1})\\
\end{align*}
This shows that the conjugate semigroup $S^{'} $ of a nearly abelian transcendental semigroup $ S $ is a nearly abelian. 

Conversely, suppose that semigroup $ S^{'} $ is nearly abelian.  Then $g_{i}\circ g_{j} = \phi \circ g_{j}\circ g_{i}$ for $ \phi \in \Phi (S) $ and for all generators $ g_{i}, g_{j}\in S^{'} $, where $ g_{i} = \phi \circ  f_{i} \circ \phi^{-1} $ and $ g_{j} = \phi \circ  f_{j} \circ \phi^{-1} $ and from which get $ f_{i} = \phi^{-1} \circ  g_{i} \circ \phi $ and $ f_{j} = \phi^{-1} \circ  f_{j} \circ \phi $. Now, for any $ f_{i}, \; f_{j} \in S $, we have
\begin{align*}
f_{i} \circ f_{j} = & (\phi^{-1} \circ  g_{i} \circ \phi) \circ (\phi^{-1} \circ  g_{j} \circ \phi)  \\
                        = & \phi^{-1} \circ  g_{i} \circ  g_{j} \circ\phi \\
                        = & \phi^{-1} \circ \phi \circ g_{j} \circ  g_{i} \circ\phi \\
                        = & g_{j} \circ g_{i} \circ \phi \\
                        = & \phi \circ  f_{j} \circ \phi^{-1}  \circ    \phi \circ  f_{i} \circ \phi^{-1} \circ \phi \\
                        = & \phi \circ f_{j} \circ f_{i}
\end{align*}
This shows that semigroup $ S $ is nearly abelian if its conjugate semigroup $ S^{'} $ is nearly abelian.

\end{proof}

To prove the theorem \ref{cgs2}, we need the following lemma.
\begin{lem}\label{cgs3}
Let $ f $ and $ g $ be two transcendental entire functions and $ \phi  $ be an entire function of the form $ z \to az + b $, where $ a \neq 0 $ such that $ \phi \circ f = g \circ \phi $. Then  $ \phi (I(f)) = I(g), \; \phi (J(f)) =  J(g) $ and $\phi (F(f)) = F(g)$.
\end{lem}
\begin{proof}
Let $ w \in \phi (I(f)) $, then there is $ z \in I(f) $ such that $ w = \phi (z) $. The condition  $ z \in I(f) \Longrightarrow f^{n}(z) \to \infty $ as $ n \to \infty $.  Now $ g^{n}(w) = g^{n}(\phi(z)) =( g^{n} \circ \phi)(z) = (g^{n-1}\circ g \circ \phi)(z) = (g^{n-1}\circ \phi \circ f)(z) =( g^{n-2}\circ \phi \circ  f^{2})(z) = \ldots = (\phi \circ f^{n})(z) = \phi (f^{n}(z)) $. Since $ \phi(z) = az + b, \, (a \neq 0) $ and  $f^{n}(z) \to \infty $ as $ n \to \infty$. So we must have $ g^{n}(w)\to \infty$ as $ n \to \infty $. This shows that $ \phi (I(f)) \subset I(g) $. For opposite inclusion, we note that if $ z \in I(g) $ then we must have $\phi( z) \in I(g) $. As above $ \phi(f^{n}(z)) = g^{n}(\phi(z)) \to \infty $ as $ n \to \infty $. This shows that $ z \in \phi (I(f)) $ and so $I(g) \subset \phi(I(f))  $. This proves that $ \phi (I(f)) = I(g)$. Remaining equality obtained from the facts $ \partial I(f) = J(f) $ and $ F(f) =\mathbb{C} \setminus J(f) $.
\end{proof} 
\begin{proof}[Proof of the Theorem \ref{cgs2}]
Let $ \phi \circ  f_{i} \circ \phi^{-1} = g_{i} $ for all $i = 1, 2, \cdots, n$. From which we get $ \phi \circ f_{i} = g_{i}\circ \phi  $ for all $i = 1, 2, \cdots, n$. Any $ f \in S $ and $ g \in S^{'} $ can be written respectively as $ f = f_{i_{1}}\circ f_{i_{2}} \circ \ldots \circ f_{i_{n}} $ and  $ g = g_{i_{1}}\circ g_{i_{2}} \circ \ldots \circ g_{i_{n}} $. From which we get $ \phi \circ f = \phi \circ f_{i_{1}}\circ f_{i_{2}} \circ \ldots \circ f_{i_{n}} = g_{i_{1}} \circ \phi \circ   f_{i_{2}} \circ \ldots \circ f_{i_{n}}  = g_{i_{1}} \circ   g_{i_{2}} \circ  \phi  \circ  \ldots \circ f_{i_{n}}  = \ldots = g_{i_{1}} \circ   g_{i_{2}} \circ   \ldots \circ g_{i_{n}} \circ \phi = g \circ \phi $ for all $ f \in S $ and $ g \in S^{'} $.   Since $S =\langle f_{1}, f_{2}, f_{3}, \ldots, f_{n} \rangle$ be a nearly abelian transcendental entire functions, so  from the {\cite[Theorem 1.1]{sub5}},  we have $I(S) = I(f),\;  J(S) = J(f) $ and ${F(S)} = {F(f)} $ for all $f \in S$.
Now $I(S) = I(f) \Longrightarrow \phi (I(S)) = \phi (I(f))$. By lemma \ref{cgs3},  $\phi(I(f)) = I(g)$. By the theorem \ref{csg1}, semigroup $ S^{'}$ is nearly abelian , so again by the {\cite[Theorem 1.1]{sub5}}, we have $ I(S^{'}) = I(g)$. Thus we get $ \phi (I(S)) = I(S^{'})$. Next two equality are also obtained by the similar fashion.
\end{proof}

\section{Proof of the Theorem \ref{cgs4}}
There is a nice way of writing arbitrary element of nearly abelian transcendental semigroup. Before doing so, we will see how does any $ f \in S $  behave just like semi-conjugacy for some member of $ \Phi(S) $ and  a member from the set generated by $ \Phi(S) $ as shown in the following lemma. Note that the statement and the proof this lemma is analogous to {\cite[Lemma 4.2]{hin}}.   
\begin{lem}\label{lemma}
Let $ S $ be a nearly abelian cancellative transcendental semigroup. Then for any $ f \in S $ and for any $ \phi \in \Phi(S) $, there is a  map $ \xi\in G $, where $ G = \langle \Phi(S) \rangle $ is a  group generated by the elements in $ \Phi(S)  $ such that $ f \circ \phi = \xi \circ f $. 
\end{lem}
\begin{proof}
For any $ \phi \in \Phi (S) $, there are $  g, h \in S $ such that $ g \circ h = \phi \circ h\circ g $. Then, for any $ f \in S $, we can write 
\begin{equation}\label{eq1}
f \circ g \circ h  =f \circ \phi \circ h\circ g.
\end{equation}
Furthermore,   
\begin{equation}\label{eq3}
f \circ g \circ h  =\xi_{1} \circ g \circ  f\circ h =\xi_{1}\circ \xi_{2} \circ f \circ  h\circ g.
\end{equation}
for some $ \xi_{1}, \xi_{2}\in \Phi(S) $. Since  $ S $ is cancellative semigroup, so
 from the equations \ref{eq1} and \ref{eq3}, we get
$$
f \circ \phi = \xi_{1}\circ \xi_{2} \circ f = \xi \circ f
$$
where $ \xi = \xi_{1}\circ \xi_{2} \in G $.
\end{proof} 
From the equation \ref{eq3}, we can say that composite of two commutators may not be a commutator. 
We investigate a couple of examples of transcendental semigroups  such that essence of above lemma \ref{lemma} holds. 
\begin{exm}
If a  semigroup $ S $ generated by functions $ f(z) = \lambda \cos z \in S$ and $ g(z) = -f(z) $, then it is nearly abelian where $ \phi (z) = -z$ is  the  commutator of $ f $ and $ g $ (see for instance  {\cite[Example 2.2]{sub5}}). Since $ \phi ^{2}(z) = z$, an identity element of a group $ G  =\langle \Phi(S) \rangle  = \{\text{Identity}, \; \phi \}$ such that  $f\; \circ \; \phi = f =\text{Identity}\;  \circ f $ and  $g\;  \circ \;  \phi = g = \text{Identity}\;  \circ g $. 
\end{exm}
\begin{exm}
If a semigroup $ S $ generated by functions $ f(z) = e^{z^{2}} + \lambda $ and $ g(z) = -f(z) $, then it is nearly abelian where $ \phi (z) = -z $ is  the  commutator of $ f $ and $ g $ (see for instance {\cite[Example 2.2]{sub5}}). Since $ \phi ^{2}(z) = z$, an identity element of a group $ G  =\langle \Phi(S) \rangle  = \{\text{Identity}, \; \phi \}$ such that  $f\; \circ \; \phi = f =\text{Identity}\;  \circ f $ and  $g\;  \circ \;  \phi = g = \text{Identity}\;  \circ g $. 
\end{exm}
Also note that in both of these examples, we have $\phi\circ f = -f \neq f = f \circ\xi   $ for any $ \xi \in G $. 
That is, for given  $ \phi \in \Phi(S) $, there may not always possible to find element $ \xi \in G $ satisfying  $\phi\circ f = f\circ\xi   $ for any choice of $ \xi \in G $. 

\begin{proof}[Proof of the Theorem \ref{cgs4}]
The proof of this theorem follows from the inductive application of above lemma \ref{lemma} to each element $ f = f_{i_{1}}\circ f_{i_{2}}\circ \ldots \circ f_{i_{n}} $ of $ S $.

\end{proof}

\end{document}